\documentclass{amsart}
\usepackage{graphicx}
\usepackage{latexsym}
\usepackage{amsfonts}
\usepackage[all]{xy}
\usepackage{amssymb, mathrsfs, amsfonts, amsmath}
\usepackage{amsbsy}
\usepackage{amsfonts}
\newtheorem*{acknowledgement}{Acknowledgement}
\newtheorem{corollary}{Corollary}

\newtheorem{definition}{Definition}
\newtheorem{lemma}{Lemma}

\newtheorem{remark}{Remark}
\newtheorem{theorem}{Theorem}
\newtheorem{example}{Example}
\numberwithin{equation}{section}

\begin{document}
\title[Einstein-type equation]{Vanishing conditions on Weyl tensor for Einstein-type manifolds}
	\author{Benedito Leandro}
\address{ Universidade Federal de Goi\'as - UFG, IME, 74690-900, Goi\^ania - GO, Brazil.}
\email{bleandroneto@ufg.br}

\keywords{Einstein-type manifolds, Weyl tensor, harmonic Weyl curvature} \subjclass[2010]{53C20, 53C21, 53C25.}
\date{\today}

\begin{abstract}
In this paper we consider an Einstein-type equation which generalizes important geometric equations, like static and critical point equations. We prove that a complete Einstein-type manifold with fourth-order divergence-free Weyl tensor and zero radial Weyl curvature is locally a warped product with $(n-1)$-dimensional Einstein fibers, provided that the potential function is proper. As a consequence, we prove a  result about the nonexistence of multiple black holes in static spacetimes. 
\end{abstract}

\maketitle
\section{Introduction}

 A smooth Riemannian manifold $M^{n}$ with smooth metric $g$ in which
\begin{eqnarray}\label{eqprincipal}
fRic=\nabla^{2}f+hg,
\end{eqnarray}
where $f,\,h: M\rightarrow\mathbb{R}$ are smooth functions is called an {\it Einstein-type manifold}. So we named \eqref{eqprincipal} {\it Einstein-type equation}. Here, $Ric$ and $\nabla^{2}$ are the Ricci tensor and the Hessian for the metric $g$, respectively.

A simple calculation from \eqref{eqprincipal} gives us
\begin{eqnarray}\label{eq1}
fR=\Delta f+nh,
\end{eqnarray}
where $R$ is the scalar cuvature for $g$ and $\Delta$ represents the Laplacian.

The reasoning behind the use of \eqref{eqprincipal} and \eqref{eq1} is that they generalize several important geometric equations: 
\begin{itemize}
	\item Static vacuum Einstein equation with null cosmological constant (cf. \cite{hwang}):
	\begin{eqnarray}\label{staticcosmnull}
		fRic=\nabla^{2}f\quad\mbox{and}\quad\Delta f=0.
	\end{eqnarray}
	
	\item Static vacuum equation with non null cosmological constant (cf. \cite{ambrozio}):
		\begin{eqnarray}\label{staticnonnullcosm}
		fRic=\nabla^{2}f+\frac{Rf}{n-1}g\quad\mbox{and}\quad\Delta f+\frac{Rf}{n-1}=0.
	\end{eqnarray}
    	\item Static perfect fluid equation (cf. \cite{shen}):
    \begin{eqnarray}\label{PFE}
    	fRic=\nabla^{2}f+\frac{(\mu-\rho)f}{n-1}g\quad\mbox{and}\quad\Delta f-\left(\frac{(n-2)\mu+n\rho}{n-1}\right)f=0,
    \end{eqnarray}
    where $\mu$ and $\rho$ are, respectively, the density and pressure smooth functions. Moreover, the energy condition implies that $\mu\geq|\rho|$.
    	\item Critical point equation (cf. \cite{baltazar}):
    \begin{eqnarray}\label{cpemetric}
    	(1+f)\mathring{R}ic=\nabla^{2}f+\frac{Rf}{n(n-1)}g\quad\mbox{and}\quad\Delta f+\frac{Rf}{n-1}=0,
    \end{eqnarray}
    where $\mathring{R}ic$ stands for the traceless Ricci tensor.
        	\item Miao-Tam equation (cf. \cite{miaotam1}):
    \begin{eqnarray}\label{miaotammetric}
    	fRic=\nabla^{2}f+\dfrac{Rf+1}{(n-1)}g\quad\mbox{and}\quad\Delta f+\frac{Rf}{n-1}=\frac{-n}{n-1}.
    \end{eqnarray}
\end{itemize}

The notion of Einstein-type manifolds is widely explored in several papers (cf. \cite{catino3}). Catino et al. \cite{catino3} provided a more general Einstein-type equation and classified it under the Bach-flat condition. However, the case of Einstein-type equation that we are assuming here was already considered by Qing and Yuan \cite{qing1} under the same Bach-flat condition, so we have taken a different approach.

In the $3$-dimensional case, Qing and Yuan \cite{qing1} proved that if the Cotton tensor is third-order divergence-free (that is, completely divergence-free), then a CPE manifold must be isometric to the round sphere, and they also get a classification for a static metric. Since the $3$-dimensional case was already considered, we have decided to focus in $n\geq4$.

Divergence conditions on Weyl have been investigated throughout the years (cf. \cite{catino2,catino,hwang,qing1,wu,hwang1}). Mathematically, this hypothesis is more natural than asymptotic flatness condition. Moreover, harmonicity has connections with conservation laws in physics.

In this paper we will explore divergence conditions on Weyl for an Einstein-type manifold satisfying \eqref{eqprincipal} and \eqref{eq1}. First, we will define a harmonic Weyl curvature when the divergence of the Weyl tensor vanishes, i.e.,
\begin{eqnarray*}
	div W=0.
	\end{eqnarray*}

In what follows, we will consider that a Riemannian manifold $(M^{n},\,g)$ has zero radial Weyl curvature  if
\begin{eqnarray}\label{zeroradialweyltensor}
W(\cdot ,\,\cdot,\,\cdot ,\,\nabla f)=0.
\end{eqnarray}
Catino \cite{catino2} used this additional hypothesis to classify generalized quasi-Einstein metrics with harmonic Weyl tensor. Furthermore, he proved that this additional hypothesis can not be removed.

Without further ado, we state our main results

\begin{theorem}\label{theocpt}
	Let $(M^{n},\,g,\,f,\,h)$, such that $n\geq4$, be a smooth compact (without  boundary) Riemannian manifold satisfying \eqref{eqprincipal} and \eqref{eq1} with zero radial Weyl curvature and fourth-order divergence-free Weyl tensor, i.e.,
	$div^{4}W=0$. Then $(M^{n}, g)$ has harmonic Weyl tensor.
\end{theorem}

As a consequence of Theorem \ref{theocpt} we have the following result, which was previously provided by Baltazar in \cite{baltazar}. Here, the demonstration is different and follows from Theorem \ref{theocpt} and Theorem 1.2 in \cite{hwang1}. It is important to remember that, if a manifold has harmonic Weyl tensor and constant scalar curvature, then its curvature is harmonic.

\begin{corollary}\label{balta}
	Let $(M^{n},\,g,\,f)$, $n\geq4$, be a CPE metric \eqref{cpemetric} with zero radial Weyl curvature satisfying
	$div^{4}W=0$. Then, $(M^{n}, g)$ is isometric to a standard sphere.
\end{corollary}

Now we will concentrate our efforts to analyze the noncompact Einstein-type manifold \eqref{eqprincipal} with vanishing conditions on Weyl. The demonstration follows a similar strategy used in Theorem 1.2 and Theorem 1.3  of \cite{catino}.

\iffalse
In our next result we consider the Einstein-type manifolds $(M^n,\,g,\,f,\,h)$, $n\geq4$, that admit a $C^{2}$ function $\psi:\mathbb{R}\rightarrow\mathbb{R}$, with $\psi(f)$ having compact support $K\subseteq M$ in which $K\bigcap f^{-1}(0)=\emptyset$. 
\fi

\begin{theorem}\label{theocomplete}
	Let $(M^{n},\,g,\,f,\,h)$, $n\geq4$, be a Riemannian manifold satisfying \eqref{eqprincipal} and \eqref{eq1} with zero radial Weyl curvature satisfying
$div^{4}W=0$. In case $f$ is a proper function, $(M,\,g)$ has harmonic Weyl curvature.
\end{theorem}

As a consequence of Theorem \ref{theocomplete}, we have the next theorem.

\begin{theorem}\label{teoclassifica}
	Let $(M^{n},\,g,\,f,\,h)$, $n\geq4$,  be a complete Riemannian manifold satisfying \eqref{eqprincipal} and \eqref{eq1} with zero radial Weyl curvature satisfying
$div^{4}W=0$. In case $f$ is proper, around any regular point of $f$, the manifold is locally a warped product with $(n-1)$-dimensional Einstein fibers.
\end{theorem}

Using another approach, Hwang et al. \cite{hwang2019} proved that the static triple and the CPE metric have harmonic Weyl tensor provided that the Bach tensor and the Weyl tensor are completely divergence-free. In this paper, the Einstein-type manifold we are working with reaches another cases, like perfect fluid spacetime.  So, in that sense, this piece has a broader appeal.

Let us give a physical application of our main result, Theorem \ref{theocomplete}. In fact, the following corollary is a consequence of Theorem \ref{theocomplete}, and Theorem 1 in \cite{hwang} (see also \cite{hwang2019}). Its is important to point out that in the static vacuum equations  $f>0$ in $M$ and $f=0$ only at the boundary $\partial M$ (cf. \cite{anderson,ambrozio}).

\begin{corollary}\label{multipleBH}
	Let $(M^n,\,g,\,f)$, $n\geq4$, be a static triple satisfying \eqref{staticcosmnull} with zero radial Weyl curvature and fourth-order divergence free Weyl tensor. In case $f$ is proper, there are no multiple black holes in $(M, g)$.
\end{corollary}

The static spacetime $(M,\,g)$ has no multiple black holes when its horizon $f^{-1}(0)=\partial M$ is connected.
Moreover, we can assume the static triple $(M,\,g,\,f)$ is connected and complete up to the boundary in such way that $g$ and $f$ extend smoothly to the boundary $\partial M$.

Uniqueness and multiplicity of black holes are a big deal in general relativity (cf. \cite{anderson,hwang,israel} and the references therein). It is well known that the Schwarzschild metric (a standard model for a static black hole) is a non trivial example of a static vacuum spacetime with harmonic curvature in all dimensions (cf. \cite{santos}). However, this is not the only example.

\begin{example}\label{example1}
	This is a very important example, we can find an analysis of the following metric in \cite{besse}, page 271.
	
	Let $(\Sigma,\, g_{\Sigma})$ be any compact $(n - 1)$-dimensional Einstein manifold with $Ric_{\Sigma} = (n-2)g_{\Sigma}$, and the warped product metric is determined by $$ds^{2}=dt^{2}+\dfrac{4f'(t)^{2}}{(n-2)^{2}}d\theta^{2}+f(t)^{2}g_{\Sigma}\quad\mbox{on}\quad\mathbb{R}^{2}\times\Sigma,$$
	where 
	\begin{eqnarray}\label{ex1}
	f\in[0,\,+\infty);\, f(0)=1,\, f'(t)>0\quad\mbox{and}\quad f'(t)^{2}=1-f(t)^{2-n}.
	\end{eqnarray}
 
	We have that $(\mathbb{R}^{2}\times\Sigma,\, ds^{2})$ is complete and Ricci-flat, and the spacelike hypersurface $\mathbb{R}^{+}\times\Sigma$, with metric
	\begin{eqnarray}\label{wp1}
	\tilde{g}=dt^{2}+f^{2}(t)g_{\Sigma},
	\end{eqnarray}
	is a solution to \eqref{staticcosmnull}, with $\frac{2}{n-2}f'(t)$ as a potential function. Therefore, the solution is smooth up to the horizon
	$\Sigma$, and complete away from it. Moreover, we can consult \cite{hwang} to see that $(\mathbb{R}^{+}\times\Sigma,\,\tilde{g})$ has a harmonic curvature. 
	
	Since $\Sigma$ is an Einstein manifold, the explicit formula of the Weyl tensor for a warped
	product manifold \eqref{wp1} allows us to deduce that $(\mathbb{R}^{+}\times\Sigma,\,\tilde{g})$ has zero radial Weyl curvature (cf. \cite{cao2} as a good survey of Weyl formulas).
	
	The metric $(\mathbb{R}^{+}\times\Sigma,\,\tilde{g})$ is asymptotic to the complete Euclidean cone on $(\Sigma,\,g_{\Sigma})$, but is asymptotically flat only in the case that $(\Sigma,\,g_{\Sigma})=
	\mathbb{S}^{n-1}(1)$, corresponding to the $n$-dimensional Schwarzschild metric (cf. \cite{santos}). In fact, if we set $f = r $ and express $t$ as a function of $r$, we have
	$$ds^{2}=\frac{dr^{2}}{1-r^{2-n}}+4\frac{1-r^{2-n}}{(n-2)^{2}}d\theta^{2}+r^{2}g_{\Sigma}.$$
\end{example}

In physics, the asymptotically flatness assumption is more often used. Nevertheless, mathematically the asymptotically flat assumption restricts the topology and geometry of the static spacetime outside a large compact set (cf. \cite{hwang} and the references therein). Therefore, the harmonic curvature condition can be more natural, at least, from a mathematical perspective.

\section{Background}

\

Now we interrupt our analysis a while to recall a little bit about generalized quasi-Einstein metrics.

\begin{definition}\cite{catino2}
	A complete Riemannian manifold $(M^n,\, g)$, $n \geq 3$, is a generalized quasi-Einstein manifold, if there exist
	three smooth functions $f$, $\mu$, $h$ on $M$, such that
	\begin{eqnarray}\label{gqe}
	Ric+\nabla^{2}f-\mu\,df\otimes df=h g.
	\end{eqnarray}
\end{definition}

Considering $u=e^{-f}$ and $\mu=1$, we rewrite \eqref{gqe} like
\begin{eqnarray*}
Ric-\frac{1}{u}\nabla^{2}u=h\,g,
\end{eqnarray*}
which is, essentially, equation \eqref{eqprincipal}. 

A generalized quasi-Einstein manifold with zero radial Weyl tensor and harmonic Weyl tensor must be locally a warped product with $(n-1)$-dimensional Einstein fibers (cf. \cite{catino2}).

\subsection{Structural Lemmas}

We would like to introduce this section by evoking some important formulas that we will need. Here we used the convention stablished by Cao et al. in \cite{cao}.
\begin{itemize}
\item Weyl tensor:\begin{eqnarray}\label{weyl}
W_{ijkl}&=&R_{ijkl}-\frac{1}{n-2}\left(R_{ik}g_{jl}-R_{il}g_{jk}+R_{jl}g_{ik}-R_{jk}g_{il}\right)\nonumber\\
&+&\frac{R}{(n-1)(n-2)}\left(g_{ik}g_{jl}-g_{il}g_{jk}\right). 
\end{eqnarray} 
\item Cotton tensor:\begin{eqnarray}\label{cotton}
	C_{ijk}&=&\nabla_{i}R_{jk}-\nabla_{j}R_{ik}-\frac{1}{2(n-1)}\left(\nabla_{i}Rg_{jk}-\nabla_{j}Rg_{ik}\right). \end{eqnarray} 
\item Bach tensor:\begin{eqnarray}\label{bach}
	B_{ij}&=&\frac{1}{n-3}\nabla^{k}\nabla^{l}W_{ikjl}+\frac{1}{n-2}R^{kl}W_{ikjl}. 
	\end{eqnarray} 
\end{itemize}

The Weyl tensor has the same symmetries of the curvature tensor. Moreover, the Weyl, the Cotton and the Bach tensors are totally trace-free. From a straightforward computation, we can see that the Cotton tensor $C$ satisfies:
\begin{eqnarray}\label{permut1}
C_{ijk}=-C_{jik},\quad C_{ijk}+C_{jki}+C_{kij}=0.
\end{eqnarray}

Moreover, from the definition of the Cotton tensor we can also infer
\begin{eqnarray*}
\nabla_{s}C_{ijk}=\nabla_{s}\nabla_{i}R_{jk}-\nabla_{s}\nabla_{j}R_{ik}-\frac{1}{2(n-1)}\left(\nabla_{s}\nabla_{i}Rg_{jk}-\nabla_{s}\nabla_{j}Rg_{ik}\right).
\end{eqnarray*}
Contracting over $i$ and $s$ we get
\begin{eqnarray}\label{c1}
	\nabla^{i}C_{ijk}=\Delta R_{jk}-\nabla^{i}\nabla_{j}R_{ik}-\frac{1}{2(n-1)}\left(\Delta Rg_{jk}-\nabla_{k}\nabla_{j}R\right).
\end{eqnarray}

Since from commutation formulas (cf. \cite{catino1} for instance), for any Riemannian manifold we have
\begin{eqnarray*}
\nabla_{i}\nabla_{j}R_{kl}=R_{ijks}R_{sl}+R_{ijls}R_{ks}+\nabla_{j}\nabla_{i}R_{kl}.
\end{eqnarray*}

Hence,
\begin{eqnarray*}
	\nabla_{i}\nabla_{j}R_{ki}=R_{ijks}R_{si}+R_{ijis}R_{ks}+\nabla_{j}\nabla_{i}R_{ki}.
\end{eqnarray*}
From the contracted second Bianchi identity
\begin{eqnarray}\label{schurid} \frac{1}{2}\nabla_{i}R=g^{jk}\nabla_{j}R_{ki},
\end{eqnarray}
we get

\begin{eqnarray}\label{c2}
	\nabla_{i}\nabla_{j}R_{ki}=R_{ijks}R_{si}+R_{js}R_{ks}+\frac{1}{2}\nabla_{j}\nabla_{k}R.
\end{eqnarray}

We can conclude from \eqref{c1} and \eqref{c2} that
\begin{eqnarray}\label{permut2}
\nabla^{i}C_{ijk}=\nabla^{i}C_{ikj}.
\end{eqnarray}
Then, from \eqref{permut1} we get
\begin{eqnarray}\label{divCotton}
\nabla^{i}C_{jki}=0.
\end{eqnarray}

Furthermore, the Cotton tensor is related to the Weyl tensor in the following manner:
\begin{eqnarray}\label{cottWeyl}
C_{ijk} =-\frac{n-2}{n-3}\nabla^{l}W_{ijkl}.
\end{eqnarray}
Thus, from \eqref{bach} and \eqref{cottWeyl} we obtain
\begin{eqnarray}\label{bachcotton}
B_{ij}=-\frac{1}{n-2}\nabla^{k}C_{ikj}+\frac{1}{n-2}R^{kl}W_{ikjl}.
\end{eqnarray} 
Also, it is easy to verify from \eqref{permut1} and \eqref{permut2} that Back tensor is symmetric.

Another important result was once proved by Cao and Chen (cf. Lemma 5.1 in \cite{cao}):
\begin{eqnarray}\label{divBach}
\nabla^{j}B_{ij}=\frac{n-4}{(n-2)^{2}}C_{ijk}R^{jk}.
\end{eqnarray}

Now we need to prove a fundamental equation for this work. In a Riemannian manifold $M$ it is possible to relate the curvature with a smooth function using the Ricci identity:
\begin{eqnarray}\label{ricciid}
\nabla_{i}\nabla_{j}\nabla_{k}f-\nabla_{j}\nabla_{i}\nabla_{k}f=R_{ijkl}\nabla^{l}f.
\end{eqnarray}

 From \eqref{eqprincipal}, it is easy to see that
\begin{eqnarray*}
	\nabla_{i}fR_{jk}+f\nabla_{i}R_{jk}=\nabla_{i}\nabla_{j}\nabla_{k}f+\nabla_{i}hg_{jk}.
	\end{eqnarray*}
Thus, from \eqref{ricciid} we get
\begin{eqnarray*}
	\nabla_{i}fR_{jk}+f\nabla_{i}R_{jk}=\nabla_{j}\nabla_{i}\nabla_{k}f+R_{ijkl}\nabla^{l}f+\nabla_{i}hg_{jk}.
\end{eqnarray*}
Contracting the above identity under $i$ and $k$, and using \eqref{schurid} we obtain
\begin{eqnarray*}
	\frac{1}{2}f\nabla_{j}R=\nabla_{j}\Delta f+\nabla_{j}h.
\end{eqnarray*}
Now from \eqref{eq1} we gather that
\begin{eqnarray}\label{eq2}
\nabla_{i}h=\frac{1}{n-1}\left(R\nabla_{i}f+\frac{1}{2}f\nabla_{i}R\right).
\end{eqnarray}

\begin{lemma}\label{lemma0}
	Let $(M^{n},\,g,\,f,\,h)$ be a smooth Riemannian manifold satisfying \eqref{eqprincipal} and \eqref{eq1}. Then,
	\begin{eqnarray*}
		fC_{ijk}&=&W_{ijkl}\nabla^{l}f+\frac{1}{n-2}\big(R_{jl}\nabla^{l}fg_{ik}-R_{il}\nabla^{l}fg_{jk}\big)\nonumber\\&+&\frac{R}{n-2}(\nabla_{i}fg_{jk}-\nabla_{j}fg_{ik})+\frac{n-1}{n-2}(\nabla_{j}fR_{ik}-\nabla_{i}fR_{jk}).
	\end{eqnarray*}	
\end{lemma}
\begin{proof}
	We star by taking the derivative of \eqref{eqprincipal} and using the Ricci identity \eqref{ricciid}
	\begin{eqnarray*}
		\nabla_{i}fR_{jk}-\nabla_{j}fR_{ik}+f(\nabla_{i}R_{jk}-\nabla_{j}R_{ik})=R_{ijkl}\nabla^{l}f+(\nabla_{i}hg_{jk}-\nabla_{j}hg_{ik}).
	\end{eqnarray*}
	From the above equation and \eqref{eq2} we get
	\begin{eqnarray*}
		\nabla_{i}fR_{jk}-\nabla_{j}fR_{ik}+f(\nabla_{i}R_{jk}-\nabla_{j}R_{ik})&=&R_{ijkl}\nabla^{l}f+\frac{R}{n-1}(\nabla_{i}fg_{jk}-\nabla_{j}fg_{ik})\nonumber\\
		&+&\frac{f}{2(n-1)}(\nabla_{i}Rg_{jk}-\nabla_{j}Rg_{ik}).
	\end{eqnarray*}
	Now the Cotton tensor gives us
	\begin{eqnarray*}
		fC_{ijk}&=&R_{ijkl}\nabla^{l}f+\frac{R}{n-1}(\nabla_{i}fg_{jk}-\nabla_{j}fg_{ik})+\nabla_{j}fR_{ik}-\nabla_{i}fR_{jk}.
	\end{eqnarray*}
	Finally, applying Weyl tensor \eqref{weyl} in the last identity we get the result.
\end{proof}

In what follows, we define
\begin{eqnarray}\label{tensorT}
T_{ijk}&=&\frac{1}{n-2}\big(R_{jl}\nabla^{l}fg_{ik}-R_{il}\nabla^{l}fg_{jk}\big)\nonumber\\
&+&\frac{R}{n-2}(\nabla_{i}fg_{jk}-\nabla_{j}fg_{ik})+\frac{n-1}{n-2}(\nabla_{j}fR_{ik}-\nabla_{i}fR_{jk}).
\end{eqnarray}
This tensor shares the same symmetries of the Cotton tensor. Thus, we can infer that
\begin{eqnarray}\label{eq4}
fC_{ijk}=W_{ijkl}\nabla^{l}f+T_{ijk}.
\end{eqnarray}

Hereafter, we assume $f(x)\neq0$ for any $x\in M$. However, it will not be an issue when proving our results and this will become clear ahead.

\begin{lemma}\label{l1}
	Let $(M^{n},\,g,\,f,\,h)$ be a smooth Riemannian manifold satisfying \eqref{eqprincipal} and \eqref{eq1}. Then, we have
\begin{eqnarray*}
(n-2)B_{ij}=-\nabla^{k}\left(\frac{T_{ikj}}{f}\right)+\frac{n-3}{n-2}\frac{C_{jki}\nabla^{k}f}{f}+\frac{W_{ikjl}\nabla^{k}f\nabla^{l}f}{f^{2}}.
\end{eqnarray*}
\end{lemma}
\begin{proof}
	From \eqref{bachcotton} and \eqref{eq4} we have
	\begin{eqnarray*}
(n-2)B_{ij}&=&-\nabla^{k}C_{ikj}+R^{kl}W_{ikjl}\nonumber\\
&=&-\nabla^{k}\left(\frac{T_{ikj}}{f}+\frac{W_{ikjl}\nabla^{l}f}{f}\right)+R^{kl}W_{ikjl}\nonumber\\
&=&-\nabla^{k}\left(\frac{T_{ikj}}{f}\right)-\frac{\nabla^{k}W_{ikjl}\nabla^{l}f}{f}+\frac{W_{ikjl}\nabla^{k}f\nabla^{l}f}{f^{2}}\nonumber\\
&-&\frac{W_{ikjl}\nabla^{k}\nabla^{l}f}{f}+R^{kl}W_{ikjl}.
\end{eqnarray*}
Now, using \eqref{eqprincipal} and admiting that the Weyl tensor is trace-free we get
\begin{eqnarray*}
(n-2)B_{ij}&=&-\nabla^{k}\left(\frac{T_{ikj}}{f}\right)-\frac{\nabla^{k}W_{ikjl}\nabla^{l}f}{f}+\frac{W_{ikjl}\nabla^{k}f\nabla^{l}f}{f^{2}}.
\end{eqnarray*}
Finally, from \eqref{cottWeyl} the result follows.
\end{proof}

\begin{lemma}\label{lemm2orderderivative}
	Let $(M^{n},\,g,\,f,\,h)$ be a smooth Riemannian manifold satisfying \eqref{eqprincipal} and \eqref{eq1}. Then, we have
\begin{eqnarray*}
C_{jki}R^{ik}=(n-2)\nabla^{i}\nabla^{k}\left(\dfrac{T_{ikj}}{f}\right)-(n-2)\dfrac{1}{f}W_{ikjl}(R^{ik}\nabla^{l}f+R^{il}\nabla^{k}f).
\end{eqnarray*}
\end{lemma}
\begin{proof}
	From Lemma \ref{l1} we have
	\begin{eqnarray*}
	(n-2)\nabla^{i}B_{ij}&=&-\nabla^{i}\nabla^{k}\left(\frac{T_{ikj}}{f}\right)+\frac{n-3}{n-2}\frac{\nabla^{i}C_{jki}\nabla^{k}f}{f}+\frac{\nabla^{i}W_{ikjl}\nabla^{k}f\nabla^{l}f}{f^{2}}\nonumber\\
	&+&\frac{n-3}{n-2}C_{jki}\left(\frac{\nabla^{i}\nabla^{k}f}{f}-\frac{\nabla^{i}f\nabla^{k}f}{f^{2}}\right)\nonumber\\
	&+&W_{ikjl}\left[\frac{1}{f^{2}}\left(\nabla^{i}\nabla^{k}f\nabla^{l}f+\nabla^{k}f\nabla^{i}\nabla^{l}f\right)-\frac{2}{f^{3}}\nabla^{i}f\nabla^{k}f\nabla^{l}f\right].
    \end{eqnarray*}
	
Next, since the Cotton tensor is trace-free, from \eqref{eqprincipal} and \eqref{divCotton} we get
	\begin{eqnarray*}
	(n-2)\nabla^{i}B_{ij}&=&-\nabla^{i}\nabla^{k}\left(\frac{T_{ikj}}{f}\right)+\frac{n-3}{n-2}C_{jki}R^{ik}-\frac{(n-3)}{f^{2}(n-2)}C_{jki}\nabla^{k}f\nabla^{i}f\nonumber\\
	&-&\frac{\nabla^{i}W_{jlki}\nabla^{k}f\nabla^{l}f}{f^{2}}+\frac{1}{f}W_{ikjl}\left[\frac{1}{f}\left(\nabla^{i}\nabla^{k}f\nabla^{l}f+\nabla^{k}f\nabla^{i}\nabla^{l}f\right)\right].
	\end{eqnarray*}

   Furthermore, from \eqref{cottWeyl} we have
    \begin{eqnarray*}
	(n-2)\nabla^{i}B_{ij}&=&-\nabla^{i}\nabla^{k}\left(\frac{T_{ikj}}{f}\right)+\frac{n-3}{n-2}C_{jki}R^{ik}+\frac{1}{f}W_{ikjl}(R^{ik}\nabla^{l}f+R^{il}\nabla^{k}f).\nonumber\\
\end{eqnarray*}

Thus, from \eqref{divBach} the result follows.
\end{proof}

\begin{lemma}\label{lemm3orderderivative}
	Let $(M^{n},\,g,\,f,\,h)$ be a smooth Riemannian manifold satisfying \eqref{eqprincipal} and \eqref{eq1}. Then, we have
	\begin{eqnarray*}
		\frac{1}{2}|C|^{2}+R^{ik}\nabla^{j}C_{jki}&=&(n-2)\nabla^{j}\nabla^{i}\nabla^{k}\left(\dfrac{T_{ikj}}{f}\right)\nonumber\\
		&-&(n-2)\nabla^{j}\left[\dfrac{1}{f}W_{ikjl}(R^{ik}\nabla^{l}f+R^{il}\nabla^{k}f)\right].
	\end{eqnarray*}
\end{lemma}
\begin{proof}
From Lemma \ref{lemm2orderderivative} we have
\begin{eqnarray*}
	C_{jki}\nabla^{j}R^{ik}+\nabla^{j}C_{jki}R^{ik}&=&(n-2)\nabla^{j}\nabla^{i}\nabla^{k}\left(\dfrac{T_{ikj}}{f}\right)\nonumber\\
	&-&\nabla^{j}\left[(n-2)\dfrac{1}{f}W_{ikjl}(R^{ik}\nabla^{l}f+R^{il}\nabla^{k}f)\right].
\end{eqnarray*}
Furthermore, from the symmetries of the Cotton tensor we obtain
\begin{eqnarray*}
	2C_{jki}\nabla^{j}R^{ik}=C_{jki}\nabla^{j}R^{ik}+C_{kji}\nabla^{k}R^{ij}=C_{jki}(\nabla^{j}R^{ik}-\nabla^{k}R^{ij}).
\end{eqnarray*}

Hence,
\begin{eqnarray*}
	\frac{1}{2}C_{jki}(\nabla^{j}R^{ik}-\nabla^{k}R^{ij})+\nabla^{j}C_{jki}R^{ik}&=&(n-2)\nabla^{j}\nabla^{i}\nabla^{k}\left(\dfrac{T_{ikj}}{f}\right)\nonumber\\
	&-&\nabla^{j}\left[(n-2)\dfrac{1}{f}W_{ikjl}(R^{ik}\nabla^{l}f+R^{il}\nabla^{k}f)\right].
\end{eqnarray*}
Since the Cotton tensor is trace-free, the result follows.
	\end{proof}

\section{Proof of the Main Results}

In what follows we will prove an integral theorem (cf. Theorem 4.1 in \cite{catino}, see also Proposition 2.3 in \cite{qing1}).

\begin{theorem}\label{theo1}
		Let $(M^{n},\,g,\,f,\,h)$, $n\geq4$, be a smooth Riemannian manifold satisfying \eqref{eqprincipal} and \eqref{zeroradialweyltensor}. For every $\psi:\mathbb{R}\rightarrow\mathbb{R}$, $C^{2}$ function with $\psi(f)$ having compact support $K\subseteq M$ in which $K\bigcap f^{-1}(0)=\emptyset$, one has
		\begin{eqnarray*}
		\frac{1}{2(n-1)}\int_{M}|C|^{2}\psi(f)=-\frac{n-2}{n-3}\int_{M}\frac{\psi(f)}{f}\nabla^{k}f\nabla^{i}\nabla^{j}\nabla^{l}W_{jkil}.
		\end{eqnarray*}
\end{theorem}
\begin{proof}
	From Lemma \ref{lemm3orderderivative} we have
		\begin{eqnarray*}
		\frac{1}{2}|C|^{2}\psi(f)+\psi(f)R^{ik}\nabla^{j}C_{jki}&=&(n-2)\psi(f)\nabla^{j}\nabla^{i}\nabla^{k}\left(\dfrac{T_{ikj}}{f}\right)\nonumber\\
		&-&(n-2)\psi(f)\nabla^{j}\left[\dfrac{1}{f}W_{ikjl}(R^{ik}\nabla^{l}f+R^{il}\nabla^{k}f)\right].
	\end{eqnarray*}

Integrating the above equation leads us to
		\begin{eqnarray*}
	\frac{1}{2}\int_{M}|C|^{2}\psi(f)+\int_{M}\psi(f)R^{ik}\nabla^{j}C_{jki}&=&-(n-2)\int_{M}\psi'(f)\nabla^{j}f\nabla^{i}\nabla^{k}\left(\dfrac{T_{ikj}}{f}\right)\nonumber\\
	&+&(n-2)\int_{M}\psi'(f)\nabla^{j}f\left[\dfrac{1}{f}W_{ikjl}(R^{ik}\nabla^{l}f+R^{il}\nabla^{k}f)\right].
\end{eqnarray*}
Since $W(\cdot,\,\cdot,\,\cdot,\,\nabla f)=0$, from Lemma \ref{lemm2orderderivative} we obtain
		\begin{eqnarray*}
	\frac{1}{2}\int_{M}|C|^{2}\psi(f)+\int_{M}\psi(f)R^{ik}\nabla^{j}C_{jki}&=&-\int_{M}\psi'(f)\nabla^{j}fC_{jki}R^{ik}.
\end{eqnarray*}
Moreover, the Cotton tensor is trace-free. Hence, from \eqref{eqprincipal} we get
		\begin{eqnarray*}
	\frac{1}{2}\int_{M}|C|^{2}\psi(f)+\int_{M}\frac{\psi(f)}{f}\nabla^{i}\nabla^{k}f\nabla^{j}C_{jki}=-\int_{M}\frac{\psi'(f)}{f}\nabla^{j}fC_{jki}\nabla^{k}\nabla^{i}f.
\end{eqnarray*}

So, integrating the right-hand side of the identity above gives us
		\begin{eqnarray*}
	\frac{1}{2}\int_{M}|C|^{2}\psi(f)&+&\int_{M}\frac{\psi(f)}{f}\nabla^{i}\nabla^{k}f\nabla^{j}C_{jki}=\int_{M}\left(\frac{\psi''(f)}{f}-\frac{\psi'(f)}{f^{2}}\right)C_{jki}\nabla^{i}f\nabla^{k}f\nabla^{j}f\nonumber\\
	&+&\int_{M}\frac{\psi'(f)}{f}\nabla^{j}f\nabla^{i}f\nabla^{k}C_{jki} + \int_{M}\frac{\psi'(f)}{f}\nabla^{k}\nabla^{j}f\nabla^{i}fC_{jki}\nonumber\\
	&=&\int_{M}\frac{\psi'(f)}{f}\nabla^{j}f\nabla^{i}f\nabla^{k}C_{jki} + \int_{M}\frac{\psi'(f)}{f}\nabla^{k}\nabla^{j}f\nabla^{i}fC_{jki}.
\end{eqnarray*}
Knowing that the Hessian is symmetric, we have
\begin{eqnarray}\label{dana1} 2\nabla^{k}\nabla^{j}fC_{jki}=\nabla^{k}\nabla^{j}fC_{jki}+\nabla^{j}\nabla^{k}fC_{kji}=\nabla^{k}\nabla^{j}f(C_{jki}+C_{kji})=0.
\end{eqnarray} 
Now remember that the Cotton tensor is also skew-symmetric. Then, renaming indices we get
		\begin{eqnarray*}
	\frac{1}{2}\int_{M}|C|^{2}\psi(f)&+&\int_{M}\frac{\psi(f)}{f}\nabla^{i}\nabla^{k}f\nabla^{j}C_{jki}=\int_{M}\frac{\psi'(f)}{f}\nabla^{j}f\nabla^{i}f\nabla^{k}C_{jki}\nonumber\\
	&=&-\int_{M}\frac{\psi'(f)}{f}\nabla^{i}f\nabla^{k}f\nabla^{j}C_{jki}=-\int_{M}\frac{\nabla^{i}\psi(f)}{f}\nabla^{k}f\nabla^{j}C_{jki}\nonumber\\
	&=&\int_{M}\frac{\psi(f)}{f}\nabla^{i}\nabla^{k}f\nabla^{j}C_{jki}-\int_{M}\frac{\psi(f)}{f^{2}}\nabla^{i}f\nabla^{k}f\nabla^{j}C_{jki}\nonumber\\
	&+&\int_{M}\frac{\psi(f)}{f}\nabla^{k}f\nabla^{i}\nabla^{j}C_{jki}.
\end{eqnarray*}
Thus,

		\begin{eqnarray}\label{oro1}
	\frac{1}{2}\int_{M}|C|^{2}\psi(f)
	+\int_{M}\frac{\psi(f)}{f^{2}}\nabla^{i}f\nabla^{k}f\nabla^{j}C_{jki}=\int_{M}\frac{\psi(f)}{f}\nabla^{k}f\nabla^{i}\nabla^{j}C_{jki}.
\end{eqnarray}

Yet, an undesirable term remains. Now we can rewrite it by an integration as it follows.
\begin{eqnarray*}
	\int_{M}\frac{\psi(f)}{f^{2}}\nabla^{i}f\nabla^{k}f\nabla^{j}C_{jki}&=&-\int_{M}\left(\frac{\psi'(f)}{f^{2}}-\frac{2\psi(f)}{f^{3}}\right)\nabla^{j}f\nabla^{i}f\nabla^{k}fC_{jki}\nonumber\\
	&-&\int_{M}\frac{\psi(f)}{f^{2}}\nabla^{j}\nabla^{k}f\nabla^{i}fC_{jki}-\int_{M}\frac{\psi(f)}{f^{2}}\nabla^{i}\nabla^{j}f\nabla^{k}fC_{jki}.
\end{eqnarray*}
Thus, from \eqref{dana1} and the skew-symmetries of Cotton, we get
\begin{eqnarray*}
	\int_{M}\frac{\psi(f)}{f^{2}}\nabla^{i}f\nabla^{k}f\nabla^{j}C_{jki}=-\int_{M}\frac{\psi(f)}{f^{2}}\nabla^{i}\nabla^{j}f\nabla^{k}fC_{jki}.
\end{eqnarray*}
Since the Cotton tensor is totally trace-free, from \eqref{eqprincipal} the above equation becomes
\begin{eqnarray}\label{oro2}
	\int_{M}\frac{\psi(f)}{f^{2}}\nabla^{i}f\nabla^{k}f\nabla^{j}C_{jki}=\int_{M}\frac{\psi(f)}{f}R^{ji}\nabla^{k}fC_{kji}.
\end{eqnarray}

Now, acknowledging that $W(\cdot,\,\cdot,\,\cdot,\,\nabla f)=0$ and that $C_{ijk}$ is totally trace-free, a simple computation from \eqref{tensorT} and \eqref{eq4} yields 
\begin{eqnarray*}
	R^{ji}\nabla^{k}fC_{kji}=\frac{1}{2}C_{kji}(R^{ji}\nabla^{k}f-R^{ki}\nabla^{j}f)=-\frac{n-2}{2(n-1)}C_{kji}T^{kji}=-\frac{(n-2)f}{2(n-1)}|C|^{2}.
	\end{eqnarray*}
Hence, from \eqref{oro1}, \eqref{oro2} and the above identity we get
		\begin{eqnarray*}
\frac{1}{2(n-1)}\int_{M}|C|^{2}\psi(f)=\int_{M}\frac{\psi(f)}{f}\nabla^{k}f\nabla^{i}\nabla^{j}C_{jki}.
\end{eqnarray*}

Finally, from \eqref{cottWeyl} we get the result.
	\end{proof}

\

\noindent {\bf Proof of Theorem \ref{theocpt}:}
  Assuming $\psi(f)=f^{4}$ the integrable condition of Theorem \ref{theo1}  can be avoided. Thus, terms (in Theorem \ref{theo1}) such as   $$\int_{M}\left(\frac{\psi'(f)}{f^{2}}-\frac{2\psi(f)}{f^{3}}\right)\nabla^{j}f\nabla^{i}f\nabla^{k}fC_{jki},$$
  will be integrable.
  
  Then, since $M$ is compact, 
from Theorem \ref{theo1} an integration gives us
			\begin{eqnarray*}
	\frac{1}{2(n-1)}\int_{M}f^{4}|C|^{2}&=&-\frac{n-2}{n-3}\int_{M}f^{3}\nabla^{k}f\nabla^{i}\nabla^{j}\nabla^{l}W_{jkil}\nonumber\\
	&=&-\frac{n-2}{4(n-3)}\int_{M}\nabla^{k}f^{4}\nabla^{i}\nabla^{j}\nabla^{l}W_{jkil}\nonumber\\
	&=&\frac{n-2}{4(n-3)}\int_{M}f^{4}\nabla^{k}\nabla^{i}\nabla^{j}\nabla^{l}W_{jkil}.
\end{eqnarray*}
Now assuming $div^{4}W=0$, we have $C=0$, and from \eqref{cottWeyl} we have a harmonic Weyl tensor.

\hfill $\Box$

\

\iffalse
\begin{remark}
It is important to point out that depending on the choice of $\psi(f)$ in Theorem \ref{theo1}, the integrable condition $K\bigcap f^{-1}(0)=\emptyset$ may be nonexistent. That is, depending on the choice of $\psi(f)$, there will be no division by $f$ in the proof of Theorem \ref{theo1}. In what follows, we will consider $\psi(f)$ in the proofs of Theorem \ref{theocpt} and Theorem \ref{theocomplete} to avoid this integrable condition.
\end{remark}
\fi

\

\noindent {\bf Proof of Theorem \ref{theocomplete}:}
	Taking $\phi\in C^{3}$, a real nonnegative function with $\phi=1$ in $[0,\,s]$, $\phi'\leq0$ in $[s,\,2s]$ and $\phi=0$ in $[2s,\,+\infty]$ for a fixed $s>0$; we have that $f$ is proper. Thus, we get that $\psi(f)=f^{4}\phi(f)$ has compact support $K\subseteq M$, for $s>0$. Then, with this choice of $\psi(f)$ the integrable condition $K\bigcap f^{-1}(0)=\emptyset$ in Theorem \ref{theo1} can be avoided. Then,  by Theorem \ref{theo1} we have
	\begin{eqnarray*}
		\frac{1}{2(n-1)}\int_{M}f^{4}\phi(f)|C|^{2}&=&-\frac{n-2}{n-3}\int_{M}\phi(f)f^{3}\nabla^{k}f\nabla^{i}\nabla^{j}\nabla^{l}W_{jkil}\nonumber\\
		&=&-\frac{n-2}{4(n-3)}\int_{M}\phi(f)\nabla^{k}f^{4}\nabla^{i}\nabla^{j}\nabla^{l}W_{jkil}
	\end{eqnarray*}

Now assuming $div^{4}W=0$, integrating by parts and applying once more Theorem \ref{theo1} we obtain
	\begin{eqnarray*}
	\frac{1}{(n-1)}\int_{M}f^{4}\phi(f)|C|^{2}&=&\frac{n-2}{2(n-3)}\int_{M}\phi'(f)f^{4}\nabla^{k}f\nabla^{i}\nabla^{j}\nabla^{l}W_{jkil}\nonumber\\
	&=&-\frac{1}{4(n-1)}\int_{M}f^{5}\phi'(f)|C|^{2}.
\end{eqnarray*}
Therefore, 
	\begin{eqnarray}\label{uc1}
	\int_{M}f^{4}|C|^{2}[\phi(f)+\frac{1}{4}f\phi'(f)]=0.
\end{eqnarray}
Since, by definition, $\phi(f)=1$ in $M_{s}=\{x\in M; f(x)\leq s\}$, we have $\phi(f)+\frac{1}{4}f\phi'(f)=1$ on the compact set $M_{s}$. Hence, from \eqref{uc1} we get
	\begin{eqnarray*}
0\leq\int_{M_{s}}f^{4}|C|^{2}=0.
\end{eqnarray*}

So, $C=0$ in $M_{s}$. Thus, taking the limit ($s\rightarrow+\infty$), we have $C=0$ on $M$. So, the result follows from \eqref{cottWeyl}.
\hfill $\Box$

\

\noindent {\bf Proof of Theorem \ref{teoclassifica}:}
Consider an orthonormal frame $\{e_{1}, e_{2}, e_{3},\ldots,e_{n}\}$ diagonalizing $Ric$ at a regular point $p\in\Sigma=f^{-1}(c)$, with associated eigenvalues $R_{kk}$, $k=1,\ldots, n,$ respectively. That is, $R_{ij}(p)=R_{ii}\delta_{ij}(p)$. Since we have harmonic Weyl tensor and zero radial Weyl curvature, from Lemma \ref{lemma0} we have $T(p)=0$, i.e.,
	\begin{eqnarray}\label{22}
\nabla_{j}f[R_{jj}+(n-1)R_{ii}-R]=0,\quad\forall i\neq j.
	\end{eqnarray}
	Without lost of generalization, consider $\nabla_{i}f\neq0$ and $\nabla_{j}f=0$ for all $i\neq j$. Then we have $Ric(\nabla f)=R_{ii}\nabla f$, i.e., $\nabla f$ is an eigenvector for $Ric$. From (\ref{22}), $R_{ii}$ has multiplicity $1$ and $R_{jj}$ has multiplicity $n-1$, for all $j\neq i$. Moreover, if $\nabla_{i}f\neq0$ for at least two distinct directions, from (\ref{22}) we have that $\lambda=R_{11}=\ldots=R_{nn}$ and we also have $\nabla f$ as an eigenvector for $Ric$.

	Therefore, in any case we have that $\nabla f$ is an eigenvector for $Ric$. From the above discussion we can take $\{e_{1}=\frac{\nabla f}{|\nabla f|},e_{2},\ldots,e_{n}\}$ as an orthonormal frame for $\Sigma$ diagonalizing $Ric$. 
	
		Now from \eqref{eqprincipal} it is important to notice that
	\begin{eqnarray}\label{eqteo7}
	fR_{a\,1}=\frac{1}{2}\nabla_{a}|\nabla f|^{2}+h\nabla_{a}f;\quad a\in\{2,\ldots,n\}.
	\end{eqnarray}
	Hence, equation \eqref{eqteo7} gives us $|\nabla f|$ constant in $\Sigma$. Thus, we can express the metric $g$ in the form
	 \begin{eqnarray*}
	 	g_{ij} = \frac{1}{|\nabla f|^{2}}df^{2} + g_{ab}(f,\theta)d\theta_{a}d\theta_{b},
	 \end{eqnarray*}
	 where $g_{ab}(f, \theta)d\theta_{a}d\theta_{b}$ is the induced metric and $(\theta_{2},\,\ldots,\,\theta_{n})$ is any local coordinate system on $\Sigma$. We can find a good overview of the level set structure in \cite{cao3,cao2,cao,leandro}.

	Observe that there is no open subset $\Omega$ of $M^{n}$ where $\{\nabla f=0\}$ is dense. In fact, if $f$ is constant in $\Omega$, since $M^{n}$ is complete, we have $f$ analytic, which implies $f$ is constant everywhere. That being said, consider $\Sigma$ a connected component of the level surface $f^{-1}(c)$ (possibly disconnected) where $c$ is any regular value of the function $f$. Suppose that $I$ is an open interval containing $c$ such that $f$ has no critical points in the open neighborhood $U_{I}=f^{-1}(I)$ of $\Sigma$. For sake of simplicity, let $U_{I}$ be a connected component of $f^{-1}(I)$. Then, we can make a change of variables 
	 \begin{eqnarray*}
	 	r(x)=\int\frac{df}{|\nabla f|}
	 \end{eqnarray*}
	 such that the metric $g$ in $U_{I}$ can be expressed by
	 \begin{eqnarray*}
	 	g_{ij}=dr^{2}+g_{ab}(r,\theta)d\theta_{a}d\theta_{b}.
	 \end{eqnarray*}
	 
	 Let $\nabla r=\frac{\partial}{\partial r}$, then $|\nabla r|=1$ and $\nabla f=f'(r)\frac{\partial}{\partial r}$ on $U_{I}$. Note that $f^{\prime}(r)$ does not change
	 sign on $U_{I}$. Moreover, we have $\nabla_{\partial r}\partial r=0.$
	 
	 From \eqref{eqprincipal} and the fact that $\nabla f$ is an eigenvector of $Ric$, the second fundamental formula on $\Sigma$ is given by
\begin{eqnarray*}\label{eq555}
h_{ab}= - \langle e_{1},\,\nabla_{a}e_{b}\rangle=\frac{\nabla_{a}\nabla_{b}f}{\|\nabla f\|}=\frac{f R_{ab}-hg_{ab}}{\|\nabla f\|}=\frac{H}{n-1}g_{ab},
\end{eqnarray*}
where $H=H(r)$, since $H$ is constant in $\Sigma$. In fact, contracting the Codazzi equation 
\begin{eqnarray*}
R_{1cab}=\nabla_{a}h_{bc}-\nabla_{b}h_{ac}
\end{eqnarray*}
over $c$ and $b$, it gives
\begin{eqnarray*}
R_{1a}=\nabla_{a}(H)-\frac{1}{n-1}\nabla_{a}(H)=\frac{n-2}{n-1}\nabla_{a}(H).
\end{eqnarray*}
On the other hand, since $R_{1a}=0$ we know that $H$ is constant in $\Sigma$.

For what follows, we fix a local coordinates system
$$(x_{1},\, \ldots,\, x_{n}) = (r,\,\ldots,\, \theta_{n}) $$
in $U_{I}$, where $(\theta_{2},\ldots,\theta_{n})$ is any local coordinates system on the level surface $\Sigma_{c}$. Considering that $a, b, c,\cdots\in\{2, \ldots, n\}$, we have
\begin{eqnarray*}
h_{ab}=-g(\partial_{r},\, \nabla_{a}\partial_{b})=-g(\partial_{r}, \Gamma^{l}_{ab}\partial_{l})=\Gamma^{1}_{ab}.
\end{eqnarray*}
Now, by definition
\begin{eqnarray*}
\Gamma^{1}_{ab}=\frac{1}{2}g^{11}\left(-\frac{\partial}{\partial r}g_{ab}\right)=-\frac{1}{2}\frac{\partial}{\partial r}g_{ab}.
\end{eqnarray*}
Then,
\begin{eqnarray*}
\frac{2}{n-1}H(r)g_{ab}=\frac{\partial}{\partial r}g_{ab}.
\end{eqnarray*}
Hence, we can infer that
\begin{eqnarray*}
g_{ab}(r,\theta)=\varphi(r)^{2}g_{ab}(r_{0},\theta),
\end{eqnarray*}
where $\varphi(r)=e^{\frac{1}{n-1}\left(\int^{r}_{r_{0}}H(s)ds\right)}$ and the level set $\{r=r_{0}\}$  corresponds to the connected component $\Sigma$ of $f^{-1}(c)$ (see more details in \cite{leandro}).

 Now, we can apply the warped product structure (cf. \cite{cao2} see also the proof of Theorem 1 in \cite{leandro}). Hence, considering  $$(M^{n},\,g)=(I,\,dr^{2})\times_{\varphi}(N^{n-1},\,\bar{g})$$
 we have
 \begin{eqnarray*}
 W_{1\,a\,1\,b}=\frac{1}{n-2}\bar{R}_{a\,b}-\frac{\bar{R}}{(n-2)(n-1)}g_{a\,b}.
 \end{eqnarray*}

 Finally, since $W(\cdot,\,\cdot,\,\cdot,\,\nabla f)=0$ we obtain that $N$ is an Einstein manifold. 
\hfill $\Box$

\
\iffalse
\noindent {\bf Proof of Corollary \ref{balta}:}
The proof follows the same strategy used in Theorem \ref{theocomplete}. We construct the nonnegative function $\phi$ having compact support $K\subseteq M$ such that $K\bigcap f^{-1}(-1)=\emptyset$ (as in Theorem \ref{theocomplete}). This will be possible because the set $f^{-1}(-1)$ has zero n-dimensional measure (see Proposition 2.1 in \cite{hwang1}). Therefore, there exist $K\subseteq M$
such that $K\bigcap f^{-1}(-1)=\emptyset$.

Then, identically as in Theorem \ref{theocomplete}, we can infer that the Cotton tensor is zero. Thus, from \cite{hwang1} we get the result.
\hfill $\Box$
\fi
\

\begin{acknowledgement}
The author would like to thank Joana T\'abata for her careful reading and relevant remarks. Moreover, we would like to thanks the referee for his valuable suggestions.
\end{acknowledgement}

\

\end{document}